\documentclass[reqno]{amsart}
\usepackage{graphicx}
\usepackage{subcaption}
\usepackage{float}
\usepackage[]{algorithm2e}
\usepackage{algorithmic}
\usepackage{hyperref}
\hypersetup{
colorlinks=true,
linkcolor=red,
filecolor=magenta, 
urlcolor=cyan,
} 
\usepackage{color}
\usepackage{fancyhdr}
\newtheorem{theorem}{Theorem}[section]
\newtheorem{lemma}[theorem]{Lemma}

\theoremstyle{definition}

\newtheorem{example}[theorem]{Example}

\theoremstyle{remark}
\newtheorem{remark}[theorem]{Remark}

\DeclareMathOperator*{\argmin}{arg\,min}
\numberwithin{equation}{section}
\newcommand{\abs}[1]{\lvert#1\rvert}

\renewcommand{\L}{\mathcal{L}^2[a,b]}
\renewcommand{\H}{\mathcal{H}^1[a,b]}

\newcommand{\HH}{\mathcal{H}^2[a,b]}
\newcommand{\HHO}{\mathcal{H}^2_0[a,b]}

\newcommand{\Lnorm}[1]{\abs{\abs{#1}}_{\mathcal{L}^2}}

\newcommand{\Hnorm}[1]{\abs{\abs{#1}}_{\mathcal{H}^1}}
\newcommand{\HHnorm}[1]{\abs{\abs{#1}}_{\mathcal{H}^2}}
\newcommand{\Lprod}[2]{(#1,#2)_{\mathcal{L}^2}}
\newcommand{\Hprod}[2]{(#1,#2)_{\mathcal{H}^1}}

\newcommand{\iab}{\int_a^b}
\newcommand{\lgrad}[1]{\nabla_{\mathcal{L}^2}^{#1}}
\newcommand{\hgrad}[1]{\nabla_{\mathcal{H}^1}^{#1}}

\begin{document}

\title{Smoothing $\mathcal{L}^2$-gradients in iterative regularization}

\author{Abinash Nayak}
\address{Visiting Assistant Professor, Department of Mathematics, University of Alabama at Birmingham, University Hall, Room 4005, 1402 10th Avenue South, Birmingham AL 35294-1241, (p) 205.934.2154, (f) 205.934.9025}
\email{nash101@uab.edu; avinashnike01@gmail.com}

\subjclass{Primary 65R30, 65R32; Secondary 65R20, 65K10}
\date{\today}

\keywords{Inverse problems, Iterative regularization, Variational minimization, Ill-posed problems, Integral equations, Tikhonov regularization.}

\begin{abstract}
Connected with the rise of interest in inverse problems is the development and analysis of regularization methods, which are a necessity due to the ill-posedness of inverse problems. Tikhonov-type regularization methods are very popular in this regard. However, its direct implementation for large-scale linear or non-linear problems is a non-trivial task. In such scenarios, iterative regularization methods usually serve as a better alternative. In this paper we propose a new iterative regularization method which uses descent directions, different from the usual gradient direction, that enable a more smoother and effective recovery than the later. This is achieved by transforming the original noisy gradient, via a smoothing operator, to a smoother gradient, which is more robust to the noise present in the data. It is also shown that this technique is very beneficial when dealing with data having large noise level. To illustrate the computational efficiency of this method we apply it to numerically solve some classical integral inverse problems, including image deblurring and tomography problems, and compare the results with certain standard regularization methods, such as Tikhonov, TV, CGLS, etc.
\end{abstract}
\maketitle

\section{\textbf{Introduction}}
\subsection{Inverse Problems and Regularization:}
An inverse problem in general is a problem where the output (effext) is known but the input (source) is not, in contrast to a direct problem where we deduce the effect from the source. Mathematically, an inverse problem is often expressed as the problem of finding/estimating a $\varphi$, given $g$, which satisfies the following operator equation:
\begin{equation}\label{T Gen.}
T\varphi = g,
\end{equation}
where $T$ is also a known operator describing the underlying physical process, the domain and range of $T$ varies depending on the problem. Typically, the solution of \eqref{T Gen.} is approximated by the solution of a least-square problem, i.e.
\begin{equation}\label{LS solution}
    \mbox{minimize} \; F(\psi) := ||T\psi - g||_2^2.
\end{equation}
However, inverse problems are usually ill-posed, in the sense of violating Hadamard's third condition: ``Continuous dependence of the data", i.e., even for a slightly perturbed data $g_\delta$, such that $||{g - g_\delta}|| \leq \delta$ (usually small), the inverse recovery becomes unstable, $||{\varphi - \varphi_\delta}|| >> \delta$ (very large), due to the unboundedness of the (pseudo-) inverse operator $T^\dagger$ and the noise present in the data. To counter such instabilities or the ill-posedness of inverse problems, regularization methods have to be employed. In the last few decades, several regularization methods have been established for linear as well as nonlinear inverse problems. Broadly, there exist two kinds of regularization approaches: 

\subsection{Tikhonov-type regularization:}
Tikhonov-type regularization methods are probably the most well known regularization techniques for solving linear as well as nonlinear inverse problems (see \cite{Engl+Hanke+Neubauer,Bakushinsky+Goncharsky, Groetsch, Baumeister,Morozov_a}), where (instead of minimizing the simple least-square problem \eqref{LS solution}) one recovers a \textit{regularized solution} by minimizing a (constrained or) penalized functional
\begin{equation}\label{Tikhonov}
F(\psi;\lambda,L,\psi_0,p,q) =  ||{T\psi - g_\delta}||_p^p + \lambda ||L(\psi - \psi_0)||_q^q,
\end{equation}
(for some $p$ and $q$) where $\lambda > 0$ is a called the regularization parameter, $||g - g_\delta|| \leq \delta$ is the error norm, $\psi_0$ is an initial guess and $L$ is a regularization operator
, with the null spaces of $T$ and $L$ intersecting trivially. 
The choice of an appropriate parameter value ($\lambda_0$) is crucial here, as it balances between the data fitting term $||{T\psi - g_\delta}||_p^p$ and the regularization term $||L(\psi - \psi_0)||_q^q$. If $\lambda$ is too small then minimizing \eqref{Tikhonov}
\textit{over-fits} the noisy data (thus, leading to a \textit{noisy recovery} or, statistically speaking, a \textit{high-variance recovery}), where as, if $\lambda$ is too big then it \textit{under-fits} the data (thus, leading to an \textit{over-smooth recovery}, or statistically, a \textit{high-bias recovery}). Choosing an appropriate $\lambda_0$ is not a trivial task, especially when dealing with large-scale problems, and typically one has to compute the solutions for many different $\lambda$ values before choosing an appropriate $\lambda_0$.

\subsection{(Semi-) Iterative regularization:}\label{semi-iterative methods}
Contrary to the above approach, in a (semi-) iterative regularization method one recovers a regularized solution $\varphi^\delta$ of \eqref{T Gen.} by simply stopping the minimization process of the least-square problem \eqref{LS solution} at an appropriate (early) instance. That is, starting from an initial guess $\psi_0^\delta$, one updates the $m^{th}$ recovered solution ($\psi_m^\delta$) in a direction ($\xi_m^\delta$) such that the $(m+1)^{th}$ recovered solution ($\psi_{m+1}^\delta$) is better than the previous, meaning $||\psi_{m+1}^\delta-\varphi|| < ||\psi_m^\delta - \varphi||$, and stops the iteration after certain number of steps (i.e., $m  \leq m_0$, for some appropriate $m_0$, usually depending on the noise level $\delta$). Mathematically,
\begin{equation}\label{steepest descent}
    \psi_{m+1}^\delta = \psi_m^\delta + \tau_m \xi_m^\delta, \hspace{1cm} 0\leq m \leq m_0(\delta) < \infty,
\end{equation}{}
where $\tau_m$ is the m-th step-length and the descent direction $\xi_m^\delta$ is usually the negative gradient of the least-square functional $F$, i.e., $\xi_m^\delta = -T^*(T\psi_m^\delta - g_\delta)$. When the step-size is fixed ($0<\tau_m = \tau < \frac{2}{||T^*T||}$) for all $m$, it's known as Landweber iterations (also known as Richardson iterations) and has been intensively investigated in the literature (see, \cite{Landweber, Hanke_Neubauer_Scherzer, Engl+Hanke+Neubauer, Bakushinsky+Goncharsky, Hanke1991, Schock1986}). The main drawback of Landweber iterations is its slow performance, i.e., it takes a large number of iterations to obtain the optimal convergence rates. To circumvent this drawback many extensions (known as polynomial or accelerated Landweber methods) have been proposed and studied in the framework of regularization (for an overview, see \cite{Hanke1991, Schock1986, Bakushinskii1979, Bakushinsky+Kokurin2004}). Such iterative methods are typically known as \textit{semi-iterative methods}, since when dealing with a noisy data ($g_\delta$) one encounters a semi-convergent nature of the recovery process. 
The advantage of these extended semi-iterative methods over the simple Landweber iteration is that, while Landweber iteration \eqref{steepest descent} uses only the last iterate $\psi_m^\delta$ to construct the new approximate $\psi_{m+1}^\delta$, in a semi-iterative method one make use of the last few iterates (if not all), 
\begin{gather}\label{Krylov-space method}
    \psi_{m+1}^\delta = \sum_{i=0}^m \mu_{m,i}\psi_i^\delta + \tau_m \xi_m^\delta,  \\
    \sum_{i=0}^m \mu_{m,i} = 1, \quad \tau_m \neq 0, \notag
\end{gather}{}
where $\xi_m^\delta$ is the descent direction (usually $\xi_m^\delta = -T^*(T\psi_m^\delta - g_\delta)$). Note that, $\psi_{m+1}^\delta - \psi_0^\delta$ belongs to the \textit{Krylov-subspace} $\mathcal{K}_{m+1}(T^*T,\xi_0^\delta)$, which is defined as
\begin{equation}\label{Krylov-subspace}
    \mathcal{K}_{m+1}(T^*T,\xi_0^\delta) := \mbox{span}\{ \xi_0^\delta, (T^*T)\xi_0^\delta, \cdots , (T^*T)^m \xi_0^\delta \}.
\end{equation}{}
Hence, these methods are also called as \textit{Krylov-subspace} methods. 
Even further acceleration is possible by adapting a conjugate-gradient type method (see \cite{Hanke_2017}), where $\mu_{m,i}$ in \eqref{Krylov-space method} depend on data $g_\delta$ (making it a non-linear method).

\subsection{A new iterative regularization method:} 
First note that, starting from the simple Landweber iterations \eqref{steepest descent}, all the generalizations and extensions \eqref{Krylov-space method} developed in the iterative regularization literature is focused only on improving the \textit{speed of the convergence} of the descent process. In this paper, we discuss a new iterative regularization method that not only improves the descent rate but also improves the smoothness of the recovery, especially when dealing with data having large noise level, and hence, leads to a much effective and efficient recovery. Note that, in all of the aforementioned iterative methods there is a direct influence of the noisy data $g_\delta$ in the descent direction, i.e., $\xi_m^\delta = T^*(g_\delta - T\psi_m^\delta)$. Therefore, when the noise level $\delta$ is large then the recovered solution $\varphi^\delta$ (though regularized) sometimes still possesses some of the noisy characteristics arising from $g_\delta$, especially when $\delta$ is large and $T^*$ does not compensate its noisy influence, see Example \ref{Example Numerical Derivative}. Here, we present a technique to pre-condition the gradient, using a smoothing operator, so that it significantly reduces the noisy influence of $g_\delta$, and thus, enhance the smoothness and accuracy of the recovery. First we observe that, a typical real-life data is usually contaminated by additive noise of \textit{zero mean}, i.e., $g_\delta = g + \epsilon_\delta$, where $\epsilon_\delta$ is a random variable such that $\mathbb{E}(\epsilon_\delta) = 0$ and $ ||\epsilon_\delta||_2^2 = ||g_\delta - g||_2^2 \leq \delta^2$. Also, notice that integrating the noisy data smooths out the noise present in the data, for example, see Figure \ref{Noisy g} ($g_\delta$ vs. $g$, with a relative error $\frac{||g - g_\delta||_2}{||g||_2}\% \approx 10\%$) vs. Figure \ref{Noisy u} ($\int g_\delta$ vs. $\int g$, with $\frac{||\int (g - g_\delta)||_2}{||\int g||_2}\% \approx 0.73\%$). This provides a heuristic motivation to incorporate the integrated data (in a regularized manner) during the minimization process to improve the smoothness of the recovery, the theoretical justifications and implementations are demonstrated in the later sections.  

\begin{figure}
    \begin{subfigure}{0.4\textwidth}
        \includegraphics[width=\textwidth]{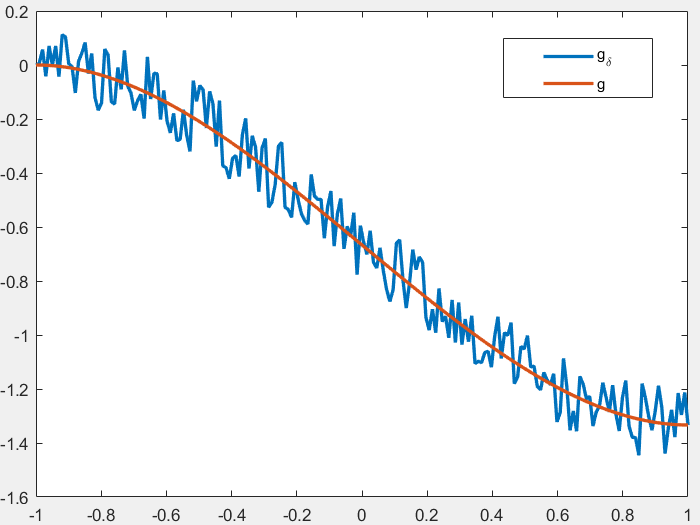}
        \caption{$g$ vs. $g_\delta$}
        \label{Noisy g}
    \end{subfigure}
    \begin{subfigure}{0.4\textwidth}
        \includegraphics[width=\textwidth]{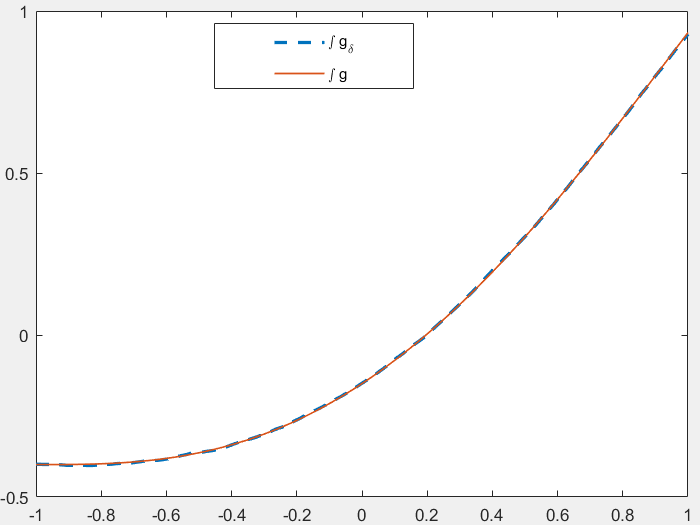}
        \caption{$\int g$ vs. $\int g_\delta$}
        \label{Noisy u}
    \end{subfigure}
    \caption{Integration smooths out the noise in the data.}  
    \label{integrated data}
\end{figure}

We briefly outline our new method here. Let's assume the operator $T: \mathcal{D}_T \rightarrow \mathcal{H}_2$ as a linear bounded injective\footnote{for simplicity, we work in an unique solution scenario. For a non-injective $T$ one can look for the minimal norm solution or the generalized solution (pseudo-inverse solution).} operator, where $\mathcal{D}_T \subset \mathcal{H}_1$ and $g \in R(T)$
is a known function defined on $[a, b] \subset \mathbb{R}$, a bounded set. Here $\mathcal{H}_1$ and $\mathcal{H}_2$ are Hilbert spaces, that we consider to be subsets of $\L$. Then an approximate solution of problem \eqref{T Gen.} is attained in an iterative fashion using the following descent direction
\begin{equation}\label{smoothed gradient}
    \xi_m^\delta = T^*(I - \Delta^{-1})(g_\delta - T\psi_m^\delta),
\end{equation}
where the integral operator $-\Delta^{-1}$ is defined as follows: for any $f \in \L$, $y = -\Delta^{-1}(f)$ is the solution of the following boundary value problem
\begin{gather}\label{operator L}
    -y'' = f,\\
    y(a) = 0, \;y(b)=0. \notag
\end{gather}
Similar to \eqref{LS solution}, we show that the Variational formulation, related to the gradient defined in \eqref{smoothed gradient}, corresponds to the minimization of the following functional
\begin{eqnarray}\label{G Gen.}
G(\psi) = \Lnorm{T\psi - g_\delta}^2 + \Lnorm{ u_\psi' - u_\delta'}^2,
\end{eqnarray}
where $u_\psi := -\Delta^{-1}(T\psi)$ and $u_\delta := -\Delta^{-1}(g_\delta)$, and the operator $-\Delta^{-1}$ is as defined above in \eqref{operator L}. Equivalently, 
\begin{align}\label{uH20}
    u_\delta(x) := (-\Delta^{-1}g_\delta)(x) = \int_x^b \int_a^\eta g_\delta(\xi) \; d\xi d\eta \; - \; \frac{b-x}{b-a}\; \int_a^b \int_a^\eta g_\delta(\xi)\; d\xi d\eta. 
\end{align}

For the exact $g$, the first term in equation \eqref{G Gen.} is minimized by the solution ($\varphi$) of the inverse problem and in \S \ref{Convexity Gen.} we prove that the second term also has the same minimizer, i.e., the solution of \eqref{T Gen.} is the minimizer of the functional $G$.

\begin{remark}
One may suspect that using an additional integral operator (in the second term of \eqref{G Gen.}) might increase the ill-posedness of the problem. But, as we are not using the noisy data $g_\delta$ directly for that term, rather a smoothed version of it (i.e., $u_\delta$) such that $||u_\delta - u|| \approx 0$, the ill-posedness of the problem is not deteriorated. In addition, the first part (involving the original operator $T$ and data $g_\delta$) together with a stopping criterion depending on $g_\delta$ also influence the recovery of a regularized solution ($\varphi^\delta$). In fact, we will show that with proper integration techniques and using both the terms one can attain even lower relative errors during the descent process than using only the first term and can also significantly reduce the semi-convergent nature of the relative errors, see Figures \ref{relative errors Deltax}, \ref{Hubble error descent} and \ref{Tomography error descent}.
\end{remark}

\begin{remark}
Note that, since (in practice) we only deal with discrete data sets, to solve an inverse problem in higher dimension ($n \geq 2$) one can convert the n-dimensional array to an one dimensional vector and apply the techniques developed for one dimensional problems. Hence, in this paper we develop the theories only for 1-dimensional problems and apply it solve n-dimensional problems ($n\geq 2$), like image deblurring and tomography, see Examples \ref{Example Image Deblurring} and \ref{Example Tomography}. 
\end{remark}

Here, we follow an improved descent algorithm for the inverse recovery of $\varphi^\delta$, where we first start with the normal $\mathcal{L}^2$-gradient of $G$ and then upgrade it to various other gradients, including the $\mathcal{H}^1$-gradient, to enhance the descent rate and efficiency of the recovery. Using an appropriate gradient is very crucial in the optimization process as it helps to retrieve the features of $\varphi$ more accurately, for example, the $\mathcal{H}^1$-gradient not only smooths the noisy $\mathcal{L}^2$-gradient but also helps in pre-conditioning certain desired boundary effects, depending on some prior information of the boundary data, (see Example \ref{Example Numerical Derivative}), details discussed in \S \ref{G-descent Gen.}.

In \S \ref{Convexity Gen.} we prove the convexity (and some other properties) of the functional $G$. In \S \ref{G-descent Gen.} we provide a descent algorithm to minimize the functional $G$ by using different gradients (or descent directions) depending on the scenarios, which is crucial for the minimization process. The convergence of the sequence of functions constructed during the descent process to a regularized solution $\varphi^\delta$ and its stability are discussed in \S \ref{Convg, Stab, Error Gen.}. 
To validate the numerical viability of the developed theory we perform numerous computational experiments, provided in \S \ref{Numerical Results}, on some classical inverse problems, like Fredholm type and Volterra type integral equations, and compare our results with the results obtained using certain standard regularization methods, like Tikhonov-type methods, CGLS, LSQR, TV and more, see Tables \ref{Table Numerical Derivative} and \ref{Imaging error comparison}.

\section{\textbf{Notations and Preliminaries}}
We adopt the following notations that are used throughout the paper. All functions are real-valued defined on a bounded closed domain $[a,b] \subset \mathbb{R}$. For $1\leq p < \infty$, $\mathcal{L}^p[a,b]$ := $(\mathcal{L}^p, \abs{\abs{.}}_{\mathcal{L}^p},[a,b])$ denotes the usual Banach space of $\abs{f}^p$ integrable functions on $[a,b]$ and the space $\mathcal{L}^\infty[a,b]:=$ $(\mathcal{L}^\infty,||.||_{\mathcal{L}^\infty}, [a,b])$ contains the essentially bounded measurable functions. Likewise the Sobolev space $\mathcal{H}^q[a,b] :=$ $(\mathcal{H}^q, ||.||_{\mathcal{H}^q},[a,b])$ contains all the functions for which $f,\;,f',\;,\cdots,\;f^{(q)} \; \in$ $\L$ and the space $\mathcal{H}^q_0[a,b] := \{ f \in \mathcal{H}^q : ``f \mbox{ vanishes at the boundary"} \}$. The spaces $\L$ and $\mathcal{H}^q$ are Hilbert spaces with inner-products denoted as $\Lprod{.}{.}$ and $(.,.)_{\mathcal{H}^q}$, respectively.  

\begin{remark}\label{Re2}
Note that, integrating the  data ($g \in \L$) twice results in $u \in \HH \subset \L$. This is particularly very significant in the sense that we are able to upgrade the smoothness of the working data or information space from $\L$ to $\HH$ and hence improve the smoothness of the recovery. 
\end{remark}

\begin{remark}
Also, note that the negative Laplacian operator $-\Delta = - \frac{\partial^2}{\partial x^2}$ is a positive operator in $\L$ on $\mathcal{D}_\Delta = \HHO$. And using the positiveness of the operator $-\Delta$ on $\HHO$ we can have a bound for $G(\psi)$ as
\begin{align}\label{Gbounds Gen.}
    \left( 1 + \frac{1}{\lambda_1} \right)^{-1} \HHnorm{u - u_\psi}^2 \leq G(\psi) \leq \HHnorm{u - u_\psi}^2\;,
\end{align}
where $\lambda_1 > 0$ is the first (smallest) eigenvalue of the positive operator $-\Delta$. Inequality \eqref{Gbounds Gen.} also implies, for a sequence of functions $\{ \psi_m \} \subset \L$, $G(\psi_m) \rightarrow 0$ if and only if $\HHnorm{u - u_{\psi_m}} \rightarrow 0$, details in \S \ref{Convg, Stab, Error Gen.}.
\end{remark}

\section{\textbf{Convexity of the functional G}}\label{Convexity Gen.}
Now let us define functionals $G_1$ and $G_2$ corresponding to the first and the second terms of the functional $G$ as follows
\begin{align}
G_1(\psi) &:= \Lnorm{T\psi - g}^2, \label{G1 Gen.}\\
G_2(\psi) &:= \Lnorm{u' - u_\psi'}^2. \label{G2 Gen.}
\end{align}
We first analyze the functional $G_2$ and then extend it to the functional $G$. 

\begin{theorem}\label{G2prop Gen.}
\end{theorem}{}
\begin{enumerate}
    \item An equivalent form of $G_2$, for any $\psi \in \L$, is
    \begin{align}\label{G2equiv Gen.}
        G_2(\psi) = \int_a^b ({u'}^2 - {u_\psi '}^2) - 2(T \psi)(u - u_\psi) \; dx
    \end{align}
    \item For any two $\psi_1, \; \psi_2 \in \L$, we have
    \begin{equation}\label{G2psi1-psi2 Gen.}
        G_2(\psi_1) - G_2(\psi_2) = \int_a^b -2(T(\psi_1 - \psi_2))(u - \frac{u_{\psi_1} + u_{\psi_2}}{2}) \; dx
    \end{equation}
    \item The first G\^{a}teaux differential, at $\psi \in \L$, for $G_2$ is given by 
    \begin{equation}\label{G2prime Gen.} 
        G_2'(\psi)[h] = \int_a^b (T h)(-2(u - u_\psi)) \; dx
    \end{equation}
    where $h \in \L$. The $\mathcal{L}^2$-gradient of $G_2$, at $\psi$, is given by
    \begin{equation}\label{G2grad Gen.}
        \nabla_{\mathcal{L}^2}^\psi G_2 = T^*(-2(u - u_\psi)),    
    \end{equation}
    where the $T^*$ is the adjoint of the operator $T$.
    \item The second G\^{a}teaux differential, at any $\psi \in \L$, of $G_2$ is given by 
    \begin{equation}\label{G2doubleprime Gen.}
        G_2''(\psi)[h,k] = 2\Lprod{-\Delta^{-1}(T h)}{(T k)}
    \end{equation}
    where $h, \; k \in \L$ and $\Delta = \frac{\partial^2}{\partial x^2}$ on $\HHO$. Hence for any $\psi \in \L$, $G_2''(\psi)$ is a positive definite quadratic form. 
\end{enumerate}
Thus $G_2$ is a strictly convex\footnote{the strict convexity follows if the operator $T$ is injective, otherwise, $G_2$ is a convex functional.} functional and has a unique global minimum which is attained at $\varphi$.

First we state an ancillary result related to $u_\psi$.
\begin{lemma}\label{ueconvg Gen.}
For fixed $\psi, \; h \in \L$, we have
\begin{equation}
    \lim_{\epsilon \rightarrow 0} u_{\psi + \epsilon h} = u_\psi\;,
\end{equation}
in $\H$.
\end{lemma}

\begin{proof}
Subtracting the following equations
\begin{align*}
    - u_\psi '' &= T\psi \\
    - u_{\psi + \epsilon h}'' &= T(\psi + \epsilon h)\;,
\end{align*}
we have 
\begin{equation}
    -(u_{\psi + \epsilon h} - u_\psi)'' = \epsilon \; T h \label{Linverse}
\end{equation}
and using $u_{\psi + \epsilon h} - u_\psi \in \HHO$ we get, via integration by parts,
\begin{align*}
 \Lprod{(u_{\psi + \epsilon h} - u_\psi)'}{(u_{\psi + \epsilon h} - u_\psi)'} = \epsilon \;\Lprod{T h}{u_{\psi + \epsilon h} - u_\psi}.
\end{align*}
Now using the positiveness of the negative laplacian, we have ${\lambda_1} \Lnorm{u_{\psi+\epsilon h} - u_\psi}^2 \leq \Lnorm{(u_{\psi+\epsilon h} - u_\psi)'}^2$, where $\lambda_1>0$ is the smallest eigenvalue of $-\Delta$, and thus
\begin{align}\label{O(e)}
 &  (1 + \lambda_1^{-1})^{-1}\; \Hnorm{u_{\psi + \epsilon h} - u_\psi}^2 \leq \Lnorm{(u_{\psi + \epsilon h} - u_\psi)'}^2 = \epsilon \;\Lprod{Th}{u_{\psi + \epsilon h} - u_\psi} \notag,\\
& \Longrightarrow	(1 + \lambda_1^{-1})^{-1}\; ||{u_{\psi + \epsilon h} - u_\psi}||_{\mathcal{H}^1}  \leq \; \epsilon \; \Lnorm{Th}, \notag
\end{align}
where the last line is obtained using the Cauchy-Schwarz inequality. Hence, \\$u_{\psi + \epsilon h} \xrightarrow{ \epsilon \rightarrow 0} u_\psi$ in $\H$, since the operator $T$ is bounded and $\psi, \; h \in \L$ are fixed, which implies the right hand side is of $O(\epsilon)$. 
\end{proof}{}

\subsection*{Proof of Theorem \ref{G2prop Gen.}}
The proof of first two properties (i) and (ii) are straight forward via integration by parts and using the fact that $u - u_\psi \in \HHO$. In order to prove (3) and (4), we use Lemma \ref{ueconvg Gen.}.

    \begin{itemize}
    \item[(3)] The G\^{a}teaux derivative of the functional $G_2$ at $\psi$ in the direction of $h$, where $h \in \L$, is given by 
    \begin{equation}\label{Gpepsilon}
        G_2'(\psi)[h] = \lim_{\epsilon \rightarrow 0} \frac{G_2(\psi + \epsilon h) - G_2(\psi)}{\epsilon}.
    \end{equation}
    Now for a fixed $\epsilon > 0$, we have using \ref{Linverse},
\begin{align*}
   & \frac{G_2(\psi + \epsilon h) - G_2(\psi)}{\epsilon} \\&= \epsilon^{-1} \iab (u' - u_{\psi + \epsilon h}')^2 - (u' - u_\psi')^2 \; dx \\
    &= \epsilon^{-1} \iab (u_\psi' - u_{\psi + \epsilon h}')(2u' - (u_{\psi + \epsilon h} + u_\psi')) \; dx\\
    &= \epsilon^{-1} \iab -(u_\psi - u_{\psi + \epsilon h})''(2u - (u_{\psi + \epsilon h} + u_\psi)) \; dx\\
    &= \epsilon^{-1} \iab - \epsilon \; (Th)(2u - (u_{\psi + \epsilon h} + u_\psi))\; dx\\
    &= -\Lprod{Th}{2u - (u_{\psi + \epsilon h} + u_\psi)}.
\end{align*}
Using Lemma \ref{ueconvg Gen.}, one obtains the G\^{a}teaux derivative of $G_2$ at $\psi \in \L$ in the direction of $h \in \L$ as
\begin{equation*}
    G_2'(\psi)[h] = \Lprod{T h}{-2(u - u_\psi)}\; .
\end{equation*}
Note that  $\Lprod{T h}{-2(u - u_\psi)} = \Lprod{h}{T^*(-2(u - u_\psi))}$ for all $h \in \L$, where $T^*$ is the adjoint of the operator $T$, as $T \in \mathcal{B}(\mathcal{H}_1, \mathcal{H}_2)$. Hence, by Riesz representation theorem, the $\mathcal{L}^2$-gradient of the functional $G_2$ at $\psi$ is given by 
\begin{equation*}
    \lgrad{\psi}G = T^* (-2(u - u_\psi)).
\end{equation*}

\item[(4)] Finally, the second G\^{a}teaux derivative for the functional $G_2$ at $\psi$ is given by 
    \begin{equation}\label{Gpp}
        G_2''(\psi)[h,k] = \lim_{\epsilon \rightarrow 0} \frac{G_2'(\psi + \epsilon h)[k] - G_2'(\psi)[k]}{\epsilon}
    \end{equation}
    Again for a fixed $\epsilon > 0$, we have using \eqref{Linverse}
\begin{align*}
   & \frac{G_2'(\psi + \epsilon h)[k] - G_2'(\psi)[k]}{\epsilon} \\&= \epsilon^{-1} \iab (T k)(-2(u - u_{\psi + \epsilon h})) - (T k)(-2(u - u_\psi)) dx \\
    &= \epsilon^{-1} \iab -2(T k)(u_\psi - u_{\psi + \epsilon h}) \; dx\\
    &= \epsilon^{-1} \iab -2(T k)(\epsilon \; \Delta^{-1}(T h))\; dx \\
    &= 2 \iab (T k)(- \Delta^{-1}(T h)) \; dx\\
    &= 2 \; \Lprod{-\Delta^{-1}(Th)}{T k} \; .
\end{align*}
    Hence from (\ref{Gpp}) and letting $\epsilon \rightarrow 0$ we get
    $$G_2''(\psi)[h,k] = 2 \; \Lprod{-\Delta^{-1}(Th)}{Tk} \;.$$
    Here we can see the strict convexity of the functional $G_2$, as for any $h \in \L$, we have
    \begin{align*}
        G_2''(\psi)[h,h] &= 2\Lprod{-\Delta^{-1}(T h)}{(T h)} \\
        &= 2\Lprod{y}{-\Delta y},
    \end{align*}
    where $-\Delta y = Th$ and $y \in \HHO$ (from (\ref{Linverse})). As $-\Delta$ is a positive operator on $\HHO$, $y$ is the trivial solution iff $T h = 0$ iff $h \equiv 0$ (T being strictly convex). Thus $G_2''(\psi)$ is a positive definite form for any $\psi \in \L$. \qed
\end{itemize}

Now one can similarly prove the above properties for functional $G_1$ and extend it for functional $G$ as well, which is as follows:
\begin{theorem}\label{Gprop Gen.} \mbox{ }
\begin{enumerate}
    \item  For any two $\psi_1, \; \psi_2 \in \L$, we have
    \begin{equation}\label{Gc1-Gc2 Gen.}
        G(\psi_1) - G(\psi_2) = \int_a^b -2T(\psi_1 - \psi_2)\left[ \left( g - \frac{T\psi_1 + T\psi_2}{2} \right) + \left( u - \frac{u_{\psi_1} + u_{\psi_2}}{2} \right)  \right]dx
    \end{equation}
    \item The first G\^{a}teaux differential, at $\psi \in \L$, for G is given by 
    \begin{equation}\label{G' Gen.}
        G'(\psi)[h] = \int_a^b (T h)(-2( (g - T\psi) + (u - u_\psi))) \; dx
    \end{equation}
    where $h \in \L$. The $\mathcal{L}^2$-gradient of $G$, at $\psi$, is given by
    \begin{equation}\label{Gl2grad Gen.}
        \nabla_{\mathcal{L}^2}^\psi G = -2T^*( (g - T\psi) + (u - u_\psi)),    
    \end{equation}
    where $T^*$ is the adjoint of the operator $T$.
    \item The second G\^{a}teaux differential, at any $\psi \in \L$, of G is given by 
    \begin{equation}\label{G'' Gen.}
        G''(\psi)[h,k] = 2\Lprod{ Th -\Delta^{-1}(T h)}{T k}
    \end{equation}
    where $h, k \in \L$. Hence for any $\psi \in \L$, $G''(\psi)$ is a positive definite quadratic form. 
\end{enumerate}
\end{theorem}

\begin{section}{\textbf{The Descent Algorithm for G}}\label{G-descent Gen.}
In this section we discuss the problem of minimizing the functional $G$, via a descent method. One can guess a descent direction by looking at the truncated Taylor expansion of the functional $G$, which is  
\begin{align}\label{taylor Gen.}
G(\psi - \gamma h) - G(\psi) \approx -\gamma G'(\psi)[h],
\end{align}
for any $\psi$, $h \in \L$, and sufficiently small $\gamma >0$. Thus $G$ is minimized at $\psi$ if the direction $h$ is chosen such that $G'(\psi)[h]>0$ for an appropriate $\gamma$.

The followings are some descent directions that can make $G'(\psi)[h] > 0$.

\subsection{The $\mathcal{L}^2$-Gradient.}\label{L2-Gradient}\mbox{}\\ 
First, notice from Theorem \ref{Gprop Gen.} that at a given point $\psi \in \L$
\begin{equation}\label{Gl2gradeq Gen.}
    G'(\psi)[h] = \Lprod{h}{\lgrad{\psi}G} \;,
\end{equation}
so if we choose the direction $h = \lgrad{\psi}G := -2T^*(u - u_\psi + g - T\psi)$ at $\psi$, then $G'(\psi)[h] > 0$. Though this gradient works well in most situations, there are certain issues associated with it during
the descent process. From a theoretical point of view, if $T^*(.)(x_0) = 0$ for some $x_0 \in [a,b]$ at every step of the descent process then $\psi_{0}(x_0)$ will be invariant during the descent process and if $\psi_{0}(x_0) \neq \varphi(x_0)$ then it will lead to severe decay or fluctuation near the point $x_0$, for details see \cite{Abinash1}. For example, if $T = \int_a^x(.) dt$ is a Volterra operator then the $\mathcal{L}^2$-gradient at any $\psi$ is always zero at the end point $b$, since $T^*(.) = \int_x^b(.) dt$, which implies that at the end point $b$ the recovery will be invariant during the descent process, see Example \ref{Example Numerical Derivative}. 

\subsection{The $\mathcal{H}^1$-Gradient.}\mbox{}\\
One can circumvent the above problem by opting for the Sobolev or Neuberger gradient, $\hgrad{\psi}G$, instead. It is the solution of the following boundary value problem 
\begin{gather}\label{h1gradeq Gen.}
    -\phi'' + \phi = \lgrad{\psi}G, \notag \\
    [\phi'\phi]_a^b = 0.
\end{gather}
This provides us a gradient $\hgrad{\psi}G := \phi$, at any $\psi$, with considerably more flexibility at the boundary points $\{ a,b \}$. In particular one has
\begin{enumerate}
    \item Dirichlet Neuberger gradient : $\phi(a) = 0$ and $\phi(b) = 0 $.
    \item Neumann Neuberger gradient : $\phi'(a)=0$ and $\phi'(b) = 0$.
    \item Robin or Mixed Neuberger gradient : $\phi(a) = 0$ and $\phi'(b)=0$ or $\phi'(a)=0$ and $\phi(b)=0$.
\end{enumerate}
In addition to the flexibility at the end points, it also enables the new gradient to be a preconditioned (smoothed) version of $\lgrad{\psi}G$, as $\phi = (I - \Delta)^{-1}\;\lgrad{\psi}G$, and hence gives a superior convergence in the steepest descent algorithms when recovering a smooth function, see Example \ref{Example Numerical Derivative}. One can also exploit the flexibility of the gradient at the end points according to some prior information of $\varphi$ at the end points. For example, if prior knowledge of $\varphi(a)$ and $\varphi(b)$ are known, then one can define $\varphi_{initial}$ as the straight line joining them and use the Dirichlet Neubeger gradient for the descent. Thus the boundary data is preserved in each of the evolving $\psi_m$'s during the descent process which leads to a more efficient and faster descent than compared to the normal $\mathcal{L}^2$-gradient. Even when $\varphi|_{\{a,b\}}$ is unknown, one can use the Neumann Neuberger gradient that allows free movements at the boundary points rather than gluing to some fixed values, see Example \ref{Example Numerical Derivative}. 

\subsection{The $\mathcal{L}^2 - \mathcal{H}^1$ Conjugate Gradient.} \label{conjg. grad. sec.} \mbox{}\\
Now, one can make use of both the gradients by taking computing the standard Polak-Ribi\'{e}re conjugate gradient scheme (see \cite{Knowles2004}) to further boost the descent rate, approximately by a factor of 2. Specifically, the initial search direction at $\psi_0$ is $h_0 = \phi_0 = \hgrad{\psi_0}G$. At $\psi_m$ one can use an exact or inexact line search routine to minimize $G$ in the direction of $h_m$ resulting in $\psi_{m+1}$. Then $\phi_{m+1} = \hgrad{\psi_{m+1}}G$ and $h_{m+1} = \phi_{m+1} + \gamma_m h_m$ where
\begin{equation}\label{conjugate gradient Gen.}
    \gamma_m = \frac{\Hprod{\phi_{m+1} - \phi_m}{\phi_{m+1}}}{\Hprod{\phi_m}{\phi_m}} = \frac{\Lprod{\phi_{m+1} - \phi_m}{\lgrad{\psi_{m+1}}G}}{\Lprod{\phi_m}{\lgrad{\psi_m}G}}. 
\end{equation}

\subsection{\textbf{The line search method for the functional G}}\mbox{}\\
At any given $\psi_m \in \L$ the functional $G$ is minimzed in the gradient direction via the single variable function $f_{m}(\gamma) = G(\psi_{m+1}(\gamma))$ and using a line search minimization, where $\psi_{m+1}(\gamma)= \psi_{m} - \gamma \hgrad{\psi_{m}}G$. To bracket the minimum the initial step size $\gamma_0$ is chosen by solving the quadratic approximation of the function $f_m$ (see \cite{Abinash1}), which is given by 
\begin{equation}\label{alpha0}
    \gamma_0 = \frac{G'(\psi_m)[\phi_m]}{G''(\psi_m)[\phi_m,\phi_m]}\;,
\end{equation}
and since the expressions for $G''$ and $G'$ are known, one can easily compute $\gamma_0$.

Hence during the descent process, starting from an initial guess $\psi_0 \in \L$, we obtain a sequence of $\mathcal{L}^2$-functions $\psi_m$ for which the sequence $\{ G(\psi_m) \geq 0\}$ is strictly decreasing. In the next section, we discuss the convergence of $\{ \psi_m \}$ to $\varphi$ and the stability of the recovery.

\end{section}

\section{\textbf{Convergence and Regularization}}\label{Convg, Stab, Error Gen.}

\textbf{Exact data:} We first prove that the sequence of functions constructed during the descent process converges to the exact solution in the absence of any noise and then proves the stability of the process in the presence of noise.

\subsection{Convergence}\label{convg Gen.}
Note that, for the sequence $\{ \psi_m\}$ produced by the steepest descent algorithm (described in Section \ref{G-descent Gen.}) we have $G(\psi_m) \rightarrow 0$, as the functional $G$ is non-negative and strictly convex (with the global minimizer $\varphi$, such that G($\varphi$) = 0) and $G(\psi_{m+1})< G(\psi_m)$. The following theorems describe the convergence of the sequence \{$\psi_m$\} to $\varphi$ in $\L$.

We first prove the convergence result for $G_2$ and then extend it for $G$.
\begin{theorem}\label{convg1 Gen.}
Suppose that $\{ \psi_m \}$ is any sequence  of $\mathcal{L}^2$-functions such that the sequence $\{ G_2(\psi_m) \}$ tends to zero. Then $\{ \psi_m \}$, and the corresponding sequence $\{ g_m := T\psi_m \}$, converge in $\L$ to $\varphi$ and $g$, respectively, and the sequence $\{ u_{\psi_m} \}$ converges to $u$ in $\H$.
\end{theorem}

\begin{proof}
First note that, using the positivity of $-\Delta$ in $\L$ on $\HHO$, we have
\begin{align}\label{G2bounds}
    \left( 1 + \frac{1}{\lambda_1} \right)^{-1} \Hnorm{u - u_\psi}^2 \leq G_2(\psi) = \Lnorm{u' - u_\psi'}^2 \leq \Hnorm{u - u_\psi}^2,
\end{align}
where $\lambda_1 = \frac{\pi^2}{(b-a)^2} > 0$ is the smallest eigenvalue of $-\Delta$. This immediately gives $u_{\psi_m} \xrightarrow{s} u$ in $\H$. Now one can easily see the weak convergence for the sequences $\{ \psi_m \}$ and $\{ g_m := T\psi_m \}$ in $\L$ to $\varphi$ and $g = T\varphi$, respectively. Since $u_{\psi_m} \xrightarrow{s} u$ in $\H$, this implies $u_{\psi_m} \xrightarrow{w} u$ in $\L$. Thus, for any $\psi$ in the range of $T^*$ we have 
\begin{align}
    \Lprod{u_{\psi_m} - u}{\psi} &= \Lprod{T(\psi_m - \varphi)}{\psi} \\
    &= \Lprod{\psi_m - \varphi}{T^*\psi}, \notag
\end{align}{}
implying $\psi_m \xrightarrow{w} \varphi$ in $R(T^*)$, and since the range of $T^*$ is dense in $\L$ for the linear, bounded and injective $T$, we have $\psi_m \xrightarrow{w} \varphi$ in $\L$. 

To see the strong convergence, we first express the functional $G_2$ in terms of an associated operator, i.e. $G_2(\psi) = \Lnorm{u_\psi' - u'}^2 = \Lnorm{L(\psi) - \psi_0}^2$, for some operator $L$ such that $L(\psi) = u_\psi'$. One can show that, from \eqref{uH20}, the function $u_\psi \in \HHO$ can be expressed as, for any $\psi \in \L$ and $x \in [a,b]$,
\begin{equation}
    u_\psi(x) = \int_x^b \int_a^\eta (T\psi) (\xi) d\xi d\eta - \frac{b-x}{b-a}\int_a^b \int_a^\eta (T\psi) (\xi) d\xi d\eta,
\end{equation}{}
and hence, the operator $L(\psi) := u_\psi'$ can be expressed as
\begin{equation}\label{Lpsi}
    L(\psi)(x) = -\int_a^x (T\psi)(\xi)d\xi  + \frac{1}{b-a}\int_a^b \int_a^\eta (T\psi)(\xi) d\xi d\eta. 
\end{equation}{}
From \eqref{Lpsi} one can see that the operator $L$ is also linear and bounded in $\L$, and thus, from the convergence theories developed for Landweber iterations or steepest descent methods, see \cite{Engl+Hanke+Neubauer, Kaltenbacher_Neubauer_Scherzer}, the sequences $\{ \psi_m \}$ and $\{ g_m\}$ also converge to $\varphi$ and $g$, respectively, in $\L$.
\end{proof}{}

Similarly, one can extend the convergence results for functional $G$ as
\begin{theorem}\label{convg2 Gen.}
Suppose that $\{ \psi_m \}$ is any sequence  of $\mathcal{L}^2$-functions such that the sequence $\{ G(\psi_m) \}$ tends to zero. Then $\{ \psi_m \}$, and the corresponding sequence $\{ g_m := T\psi_m \}$, converge in $\L$ to $\varphi$ and $g$, respectively, and $\{ u_{\psi_m} \}$ converges to $u$ in $\HH$.
\end{theorem}

\subsection{Regularization (through appropriate early stopping)}
As explained earlier, when solving an ill-posed problem with noisy data, one can not recover a regularized solution by simply minimizing the associated functional ($F$ in \eqref{LS solution} or $G_1$ in \eqref{G1 Gen.}) completely. Here, the recovery error follows a semi-convergent nature, i.e., in the initial iterations the sequence $\{ \psi_m^\delta \}$ converges towards the true solution ($\varphi$) but upon further iterations $\psi_m^\delta$ diverges away from $\varphi$, due to the noise and ill-posedness of the problem. Equivalently, during the initial stage $||\psi_m^\delta - \varphi|| \rightarrow 0$ for $m \leq m_0(\delta) < \infty$ (for some appropriate $m_0(\delta)$), but upon further iterations $||\psi_m^\delta - \varphi|| \rightarrow \infty$ (or $||\psi_m^\delta|| \rightarrow \infty$), as $m \rightarrow \infty$. Hence such methods are also called as \textit{semi-iterative methods} and terminating the minimization process at an appropriate early iteration leads to a regularized solution, i.e., the iteration index ($m$) serves as a regularization parameter, for details see \cite{Engl+Hanke+Neubauer, Kaltenbacher_Neubauer_Scherzer}.

Here we show that the semi-convergent nature is also shown by the functional $G_2$ when recovering the true (or original) solution $\varphi$ but using a noisy data $g_\delta$. First we see that for an exact $g$ (or equivalently, an exact $u$) and the functional ($G$) constructed based on it, we have the true solution satisfying $G(\varphi) = 0$. However, for a given noisy $g_\delta$, with $ \Lnorm{g_\delta - g} = \delta > 0$, and the functional ($G_\delta$) based on it we will have $G_\delta (\varphi) > 0$, see Theorem \ref{well-posedness Gen.}. So if we construct a sequence of functions $\{\psi_m^\delta \} \subset \L$, using the descent algorithm and based on the noisy data $g_\delta$, such that $G_\delta(\psi_m^\delta) \rightarrow 0$ then (from Theorem \ref{convg1 Gen.} or \ref{convg2 Gen.}) we will have $\psi_m^\delta \rightarrow \varphi_\delta$, where $\varphi_\delta$ is the recovered noisy solution satisfying $G_\delta(\varphi_\delta) = 0$. This implies initially $\psi_m^\delta \rightarrow \varphi$ and then upon further iterations $\psi_m^\delta$ diverges away from $\varphi$ and approaches $\varphi_\delta$, such that $\Lnorm{\varphi - \varphi_\delta} >> \delta$. This typical behavior of any ill-posed problem is managed, as stated above, by stopping the descent process at an appropriate iteration $M(\delta)$ such that $G_\delta(\psi_{M(\delta)}^\delta) > 0$ but close to zero, this gives us a regularized solution $\varphi^\delta := \psi_{M(\delta)}^\delta$ such that $\Lnorm{\varphi^\delta - \varphi} \leq C(\delta)$.

Following a similar argument as in \eqref{Gbounds Gen.} we have a lower bound for $G_\delta(\varphi)$ as
\begin{theorem}\label{well-posedness Gen.}
Let $G$ and $G_\delta$ be the functionals defined corresponding to the given data $g$ and $g_\delta$ (or $u, u_\delta$) and, $\varphi, \varphi_\delta$ denote their respective minimizers (i.e., $G(\varphi)=0$ and $G_\delta(\varphi_\delta)=0$), then we have the following lower bound for $G_\delta(\varphi)$ 
\begin{align}\label{Glowerbounds Gen.}
    G_\delta(\varphi) \geq \left( 1 + \frac{1}{\lambda_1} \right)^{-1} \HHnorm{u - u_\delta}^2,
\end{align}
where $\lambda_1 > 0$ is the smallest eigenvalue of the positive operator $-\Delta$ on $\HHO$, and vice-versa for $G(\varphi_\delta)$.
\end{theorem}
Therefore, combining \eqref{Gbounds Gen.} and Theorem \ref{well-posedness Gen.} we have the following two-sided bound for $G_\delta(\varphi)$, for some constants $C_1$ and $C_2$,
\begin{equation}
C_1 \HHnorm{u - u_\delta}^2 \leq G_\delta(\varphi) \leq C_2 \HHnorm{u - u_\delta}^2.
\end{equation}
Thus, when $\delta \rightarrow 0$ we have $G_\delta(\varphi) \rightarrow 0$ which implies $\varphi_\delta \rightarrow \varphi$ in $\L$.

Though from \eqref{Glowerbounds Gen.} we would like to stop the descent process when
\begin{equation}\label{stopping criteria 0 Gen.}
G_\delta(\psi_m^\delta) < \left( 1 + \frac{1}{\lambda_1} \right)^{-1} \HHnorm{u - u_\delta}^2,
\end{equation}  
but as we do not know the exact $g$ (equivalently, the exact $u$) we can not use \eqref{stopping criteria 0 Gen.} directly as the stopping criteria for the descent process.

\subsection{{Stopping criterion I}}\label{stopping criteria I}
When the error norm $\delta = \Lnorm{g - g_\delta}$ is known then \textit{Morozov's discrepancy principle}, \cite{Morozov_a}, can serve as a stopping criterion for the descent process, i.e., terminate the iteration when 
\begin{equation}\label{morozov principle Gen.}
\Lnorm{T\psi_m^\delta - g_\delta} \leq \tau \delta
\end{equation}
for an appropriate $\tau > 1$\footnote{In our experiments, we used the termination condition:  $\Lnorm{T\psi_m^\delta - g_\delta} < \delta$.}. For the convergence and stability of the process see \cite{Hanke_Neubauer_Scherzer}. 

\section{\textbf{Numerical Results}}\label{Numerical Results}
We follow the following (pseudo-) Algorithm \ref{Algorithm} and perform inverse recoveries on some standard integral type equations. A MATLAB code was written to test the numerical viability of the method and the results obtained are then compared with certain standard regularization methods, 
see Table \ref{Table Numerical Derivative} and Table \ref{Imaging error comparison}. 

\begin{algorithm}[ht]
 \KwData{Given noisy $g_\delta$, one construct $u_\delta$ as defined in \ref{uH20}}
 \KwResult{ Variational recovery of a regularized solution ($\varphi^\delta$) for \eqref{T Gen.}}
 \textbf{Initialization:} Start with an initial guess ($\psi_0^\delta \equiv 0$)\\
 \While{$\Lnorm{T\psi^\delta - g_\delta} > \Lnorm{g - g_\delta}$}{
 \textbf{Choose a gradient ($\nabla G$):}\\
  \begin{itemize}
      \item $\lgrad{}G = -2T^*( (g_\delta - T\psi^\delta) + (u_\delta - u_{\psi^\delta}))$
      \item $\hgrad{}G = (I - \Delta)^{-1}\lgrad{}G$, with Neumann B.C., if $\lgrad{}G|_{\{a \; | \; b\}} = 0$
      \item Conjugate gradients ($h_m$), as defined in \S \eqref{conjg. grad. sec.}, for faster descent rate
  \end{itemize}
  \textbf{Define:} $\psi^\delta_{\gamma}$ := $@(\gamma)$ $\psi^\delta$ - $\gamma \nabla G$ and 
  $f_\gamma$ := $@(\gamma) \; G(\psi_\gamma^\delta)$\\
  Find the optimal $\gamma_0$, as defined in (\ref{alpha0})\\
  \eIf{$G(\psi_{\gamma_0}^\delta) \geq G(\psi^\delta)$}{
   minimize $f_\gamma$ in $[0, \gamma_0]$, using a 1D-minimizer (Brent minimization)\\
   $\tilde{\gamma} = \argmin_{[0,\gamma_0]} \; f_\gamma$\;
   }{
   $\gamma_1 = \gamma_0$ and $\gamma_2 = \gamma_1 + \gamma_0$\\
   \While{$f_{\gamma_2} < f_{\gamma_1}$}{
   $\gamma_1 = \gamma_2$ and $\gamma_2 = \gamma_1 + \gamma_0$
   }
   minimize $f_\gamma$ in $[\gamma_1, \gamma_2]$\\
   $\tilde{\gamma} = \argmin_{[\gamma_1,\gamma_2]} \; f_\gamma$\;
  }
  $\psi^\delta = \psi^\delta - \tilde{\gamma} \; \nabla G$\;
 }
 \caption{The Descent Algorithm}\label{Algorithm}
\end{algorithm}

\subsection{Fredholm Integral Equation of the First Kind}\label{section fredhom integral}
A Fredholm integral equation of the first kind is an integral operator equation that depends on a kernel function $K(s,t)$ and is given by
\begin{equation}\label{fredholm equation}
T\varphi := \int_a^b K(s,t) \varphi(t) dt = g(s) \;, \mbox{ for } s \in [c, d].
\end{equation}
Such integral operator equations are classical inverse problems and can be quite ill-posed. 
The operator equation \eqref{fredholm equation} is discretized using either Galerkin or Nystr$\ddot{o}$m methods to yield the following linear discrete ill-posed problem 
\begin{equation}
    Ax = b,
\end{equation}
where the matrix $A \in \mathbb{R}^{m\times n}$ is the discretized representation of the operator $T$ and the vectors $x \in \mathbb{R}^n$, the $b \in \mathbb{R}^m$ are the discretized source and effect functions. MATLAB functions in \cite{Hansen_MATLAB2020, Hansen_IRtools, Hansen2007} determine the discretizations $A \in \mathbb{R}^{n\times n}$, the scaled discrete approximations of $x \in \mathbb{R}^n$ and $b := Ax \in \mathbb{R}^m$. To test the stability of the method, a Gaussian (zero-mean) error vector $\epsilon_\delta \in \mathbb{R}^m$ is added to $b$ to get a perturbed vector $b_\delta \in \mathbb{R}^m$, such that $\frac{||\epsilon_\delta||_2}{||b||_2}$ (the relative error norm) is around $10\%$, unless otherwise stated. In particular, when using discrepancy principle to terminate the descent process we assume $\delta = ||\epsilon_\delta||_2$ to be known.

\subsection{Numerical Differentiation}
In the first example we study 
the problem of numerical differentiation (of noisy data), which is also governed by an integral equation (also known as a Volterra equation) of the following form
\begin{equation}\label{Volterra eq.}
    \int_a^x \varphi(t) dt = g(t) - g(a), 
\end{equation}{}
for $x \in [a,b]$, and comparing \eqref{Volterra eq.} to \eqref{fredholm equation}, the kernel function can be considered as $K(s,t ) = \chi_{[a,x]}(t)$, where $\chi$ is the standard characteristic function, for details see \cite{Abinash1}. Now, for some given discrete noisy data values $\{g_\delta(x_i)) \}_{i=1}^n$, we construct a discrete set of values $\{ u_\delta(x_i) \}_{i=1}^n$ for the function $u_\delta$, as defined in \eqref{uH20}. Note that, when using the formula \eqref{uH20} to construct $u_\delta$ or $u_\psi$, the values of $u_\delta(x)$ and $u_\psi(x)$ depend on the choice of the value of the integration mesh-size $\Delta x$, because of the discretization of the integral, where for any integrable function $f$,
\begin{equation}
    \int_a^x f(t)dt = \Delta x\sum\limits_{i=1}^{m: x_m=x} f(x_i).
\end{equation}
Hence, in the following examples we show that increasing the $\Delta x$ values leads to smoother solutions as well as saturating the relative errors in the recovery. 
Since the kernel function $K(s,t)$ is a simple step-function, which is not smooth, one can notice (see Figure \ref{ND graph}) that the recovered solution obtained using the discrepancy principle, though regularized, is also not smooth. However, with increasing $\Delta x$ values the smoothness of the recovery increases, see Figure \ref{recovery Deltax}.
Also, note that here (as explained in \S \ref{L2-Gradient}) using $\mathcal{L}^2$ or $\mathcal{L}^2$-$\mathcal{L}^2$ conjugate gradient wouldn't be a good choice if $\varphi(b) \neq 0$, as $(\lgrad{}G)(b)=0$ (since $T^*(b)=0$), which leads to a sharp decay at the end point `b', see Figure \ref{phib n= 0} and Table \ref{Table Numerical Derivative}. In such scenarios, either Neumann $\mathcal{H}^1$ or $\mathcal{L}^2$-$\mathcal{H}^1$ conjugate gradient will be very effective.

\begin{example}\label{Example Numerical Derivative}
We contaminate two functions $g_1(x) =\frac{x^3}{3}-x$ and $g_2(x) = \frac{x^3}{3} - \frac{x}{2}$ on $[-1,1]$ with Uniform noise, to produce noisy $g_{1,\delta}$ and $g_{2,\delta}$ with a relative noise level of 10\%. We collected evenly spread 200 of those noisy data points, which corresponds to a grid-size of 0.01. Note that, since $g_2'(x) = x^2$ on $[-1,1]$, we have $\varphi_2(b) = g_2'(1) \neq 0$. Now we compute the numerical regularized-derivatives from these noisy $g_\delta$'s data points and compare it with the true derivatives $g_i'$'s. Figure \ref{ND graph} shows the comparison of the recoveries obtained using Tikhonov regularization and our method with the $\mathcal{L}^2$-$\mathcal{L}^2$ conjugate gradient and the $\mathcal{H}^1$-$\mathcal{L}^2$ conjugate gradient, with the relative errors in Table \ref{Table Numerical Derivative}. Note that, when computing $g_{2,\delta}'$ (where $g_{2}'(b)\neq 0$) the use of Neumann $\mathcal{L}^2$-$\mathcal{H}^1$ conjugate gradient leads to a much better and smoother recovery, because of the flexibility at the boundary point $b$ and $\hgrad{}G = (1 - \Delta)^{-1}\lgrad{}G$ (leading to a smoother gradient).

In the above computations we used the integration mesh-size $\Delta x = 0.01$ for constructing $u_\delta$ and $u_\psi$. Then we repeat the experiment with the same number of noisy data points but with different $\Delta x$ values for the integration. Note that, here $\Delta x$ represents the mesh-size for the integral to get $u_\delta$ or $u_\psi$ values and is not connected to the original grid-size of the problem (which determines the number of data points). Figure \ref{recovery Deltax} shows the increase in the smoothness level with the increase in the $\Delta x$ values and Figure \ref{relative errors Deltax} shows the saturating effect shown by increasing $\Delta x$ values. In fact, increasing $\Delta x$ values not only saturates the recovery errors during the descent process but also lowers the minimum error of the recovery (though we cannot obtain that solution directly from the discrepancy principle).
\end{example}{}

\begin{figure}
    \begin{subfigure}{0.4\textwidth}
        \includegraphics[width=\textwidth]{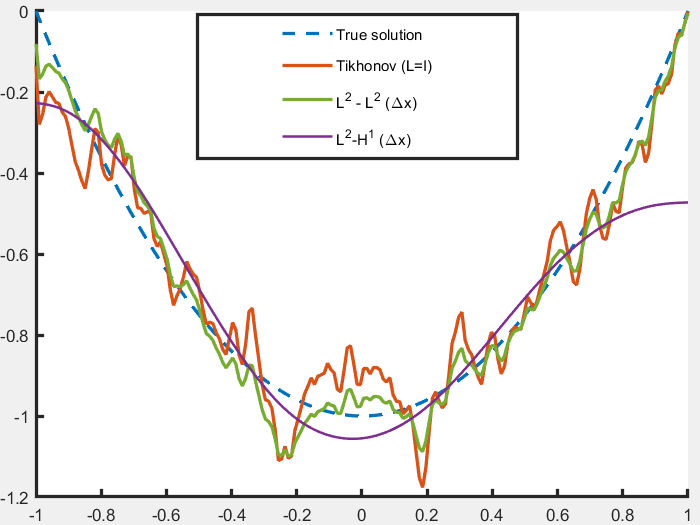}
        \caption{Computed $g_{1,\delta}'$ with $\varphi(b)=0$}
        \label{phib=0}
    \end{subfigure}
    \begin{subfigure}{0.4\textwidth}
        \includegraphics[width=\textwidth]{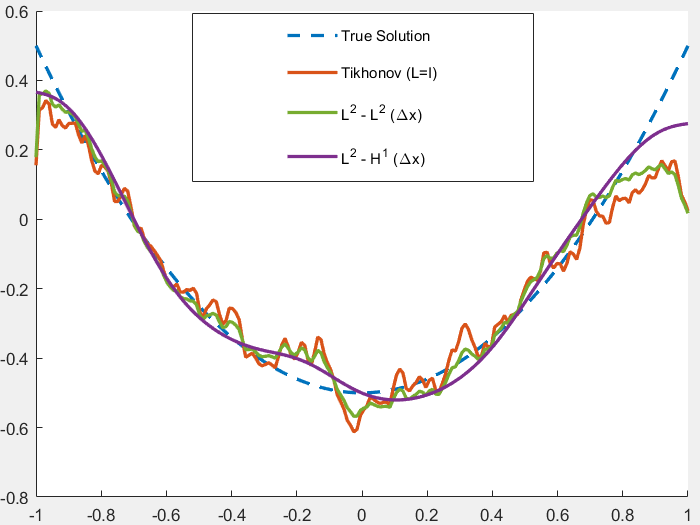}
        \caption{Computed $g_{2,\delta}'$ with $\varphi(b)\neq 0$}
        \label{phib n= 0}
    \end{subfigure}
    \caption{ (a) $g_1'(x) = x^2 - 1$, (b) $g_2'(x)= x^2 - 0.5$; Computed: Tikhonov (green), Our method: $\nabla_{\mathcal{L}^2-\mathcal{L}^2}$ (red), $\nabla_{\mathcal{L}^2-\mathcal{H}^1}$ (magenta). Observe that, using $\nabla_{\mathcal{L}^2-\mathcal{H}^1}$ leads to much smoother solutions.} 
    \label{ND graph}
\end{figure}

\begin{figure}
    \begin{subfigure}{0.4\textwidth}
        \includegraphics[width=\textwidth]{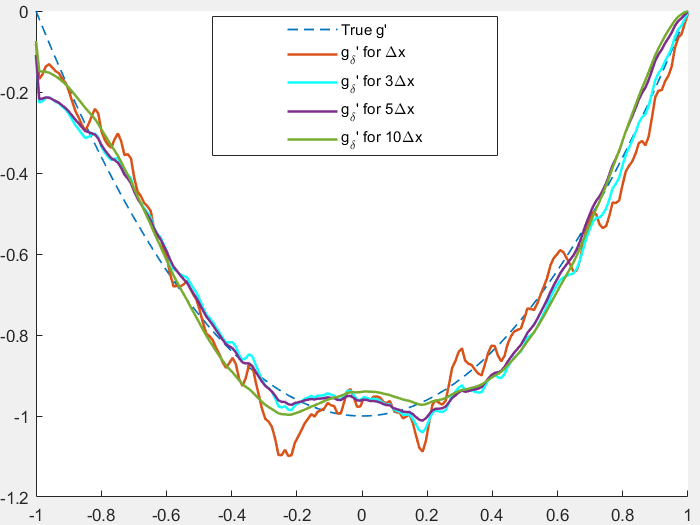}
        \caption{Computed $g_{1,\delta}'$ for various $\Delta x$'s.}
        \label{recovery Deltax}
    \end{subfigure}
    \begin{subfigure}{0.4\textwidth}
        \includegraphics[width=\textwidth]{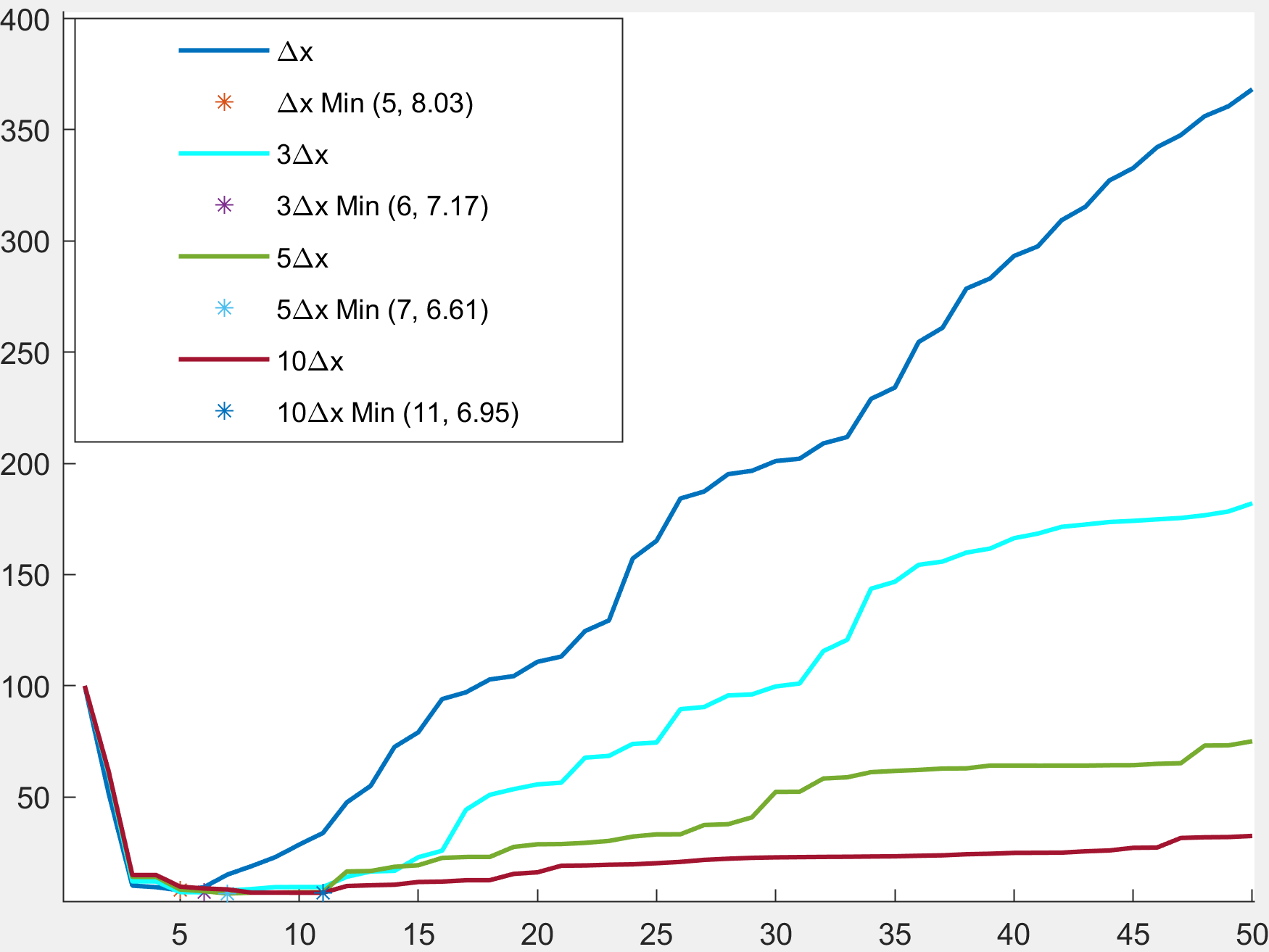}
        \caption{Relative errors for various $\Delta x$'s.}
        \label{relative errors Deltax}
    \end{subfigure}
    \caption{ (a) Recoveries for different $\Delta x$ values, (b) Recovery errors during the descent process with different $\Delta x$'s for Example \ref{Example Numerical Derivative}. Note that, the smoothness of the recoveries increases with increasing $\Delta x$ values and it also saturates the recovery errors.}  
    \label{Different Deltax}
\end{figure}

\begin{table}[ht]
    \centering
    \begin{tabular}{|p{2.1cm}||p{1.3cm}|p{1cm}|p{1cm}||p{1.3cm}|p{1cm}|p{1cm}|p{1cm}|}
    \hline
    \multicolumn{7}{|c|}{Stopping strategy: Morozov's discrepancy principle, see \eqref{morozov principle Gen.}}\\
    \hline
        & $\varphi(b)=0$ & &  & $\varphi(b) \neq 0$  &  &\\
    \hline
    Methods & $\Delta x$ & $3\Delta x$ & $5\Delta x$ & $\Delta x$ & $3\Delta x$ & $5 \Delta x$ \\
    \hline
    Tik ($L=I$)  & 0.1179 & &  & 0.2748 & & \\
    \hline
    Tik ($L=|\nabla|$) & 0.1822 &  &  & 0.2205 &   &  \\
    \hline
    TSVD & 0.0695 &  &  & 0.2734 &   &  \\
    \hline
    TTLS & 0.4544 &  &  & 0.3521 &   &  \\
    \hline
    CGLS & 0.0814 & &  &  0.2543 &  &  \\
    \hline
    LSQR$_B$ & 0.0814 & &  &  0.2543 &  &  \\
    \hline
    $\nu$-method & 0.1734 & &  &  0.2651 &  &  \\    
    \hline
    \hline
    {\color{red} G ($\nabla_{\mathcal{L}^2 - \mathcal{L}^2}$})& {\color{red} 0.0803} & {\color{red} 0.0720} & {\color{red} 0.0661} & {\color{red} 0.2574} & {\color{red} 0.2691} & {\color{red} 0.2921} \\
    \hline
    {\color{red} G ($\nabla_{\mathcal{L}^2-\mathcal{H}^1}$}) & {\color{red} 0.1608} & {\color{red} 0.1907} & {\color{red} 0.1861} & {\color{red} 0.1474} & {\color{red} 0.2237} & {\color{red} 0.2218} \\
    \hline    
    \end{tabular}
    \caption{Relative errors for Example \ref{Example Numerical Derivative}.}
    \label{Table Numerical Derivative}
\end{table}


\subsection{Imaging Problems}
In this subsection we deal with higher dimensional inverse problems. Considering an image as a matrix leads to a 2-dimensional problem. However, we can transform a two dimensional $(n\times n)$-matrix to an one dimensional vector of length $n^2$ by simply stacking either the columns or rows over each other and hence, converting the 2-dimensional problem to a 1-dimensional problem. Now, since the integral limits can be arbitrary we assign the limits based on grid or mesh size, i.e., assuming $a=0$ and some value for $\Delta x$ we have $b=n^2\Delta x$. 
Here, since the transformed 1D vector corresponding to an image is highly discontinuous and non-smooth, having a large $\Delta x$ value leads to over-smooth the solution. Hence, we start with the initial choice $\Delta x = 10^{-4}$ and show that with conservative increments one can attain a similar saturation of the relative errors as shown in Figure \ref{relative errors Deltax}. We used $\lgrad{}G$ as the descent direction during the recovery process and the dependence of the relative errors on the integration mesh-size ($\Delta x$) is projected in Figure \ref{imaging problems pics}.

\begin{example}\label{Example Image Deblurring}\textbf{Image deblurring}\\
In this example we apply our regularization technique to deblurr a noisy blurred image. The mechanism behind blurring an image also results in a Fredholm integral equation of first kind \eqref{fredholm equation}, with kernel being a Gaussian spread function, i.e., $K(s,t) = \frac{1}{\sqrt{2\pi \sigma^2}} e^{\frac{(s-t)^2}{2\sigma^2}}$, where $\sigma$ determines the spread and the ill-posedness of the problem (ill-posedness increases with the increase in the $\sigma$ value, here we considered $\sigma=3$). A true test image ($\varphi$), see Figure \ref{True_image_deblurr}, is first blurred and then contaminated with Gaussian noise ($\approx$ 10\% relative error) to get a blurred noisy image ($g_\delta$), see Figure \ref{blurred_noisy_image}. Again, the discretized matrix $A \in \mathbb{R}^{n\times n}$ and the vectors $x$, $b$ in $\mathbb{R}^n$ (where $n=16384=128\times 128$) are generated using the routine, $blur()$, from \cite{Hansen_MATLAB2020,Hansen2007}. In this problem we performed a constraint minimization\footnote{i.e., we project the negative values to zero.}, $\psi \geq 0$, during the recovery process, since the pixel values are non-negative, and used the $\lgrad{}G$ gradient as the descent direction. Figure \ref{recovered_deblurred_image} shows the recovered image (using functional $G$) and Table \ref{Imaging error comparison} shows an error comparison between our recovery and the recoveries obtained by using some of the standard iterative regularization methods such as Least Squares minimization with (i) $\mathcal{L}^1$-norm (IRell1), (ii) (heuristic) total variation penalization (IRhtv), (iii) Conjugate Gradient (CGLS), (iv) FISTA for constrained LS (IRfista), (v) modified residual norm steepest descent (IRmrnsd), (vi) enriched CGLS (IRenrich), (vii) $\mathcal{L}^1$-norm penalization term (IRirn), (viii) hybrid FGMRES for enforcing $\mathcal{L}^1$-norm penalization (IRhybrid\_fgmres), (ix) hybrid GMRES for square systems (IRhybrid\_gmres), (x) hybrid LSQR (IRhybrid\_lsqr), (xi) restarted Krylov subspace method (IRrestart), where the routines for the above mentioned methods are taken from \cite{Hansen_IRtools}, with the non-negative constraint (i.e., \textit{`nonnegativity',`on')} and the stopping rule as the discrepancy principle (\textit{`stoprule',`DP'}).

\begin{figure}[ht]
    \centering
    \begin{subfigure}{0.4\textwidth}
        \includegraphics[width=\textwidth]{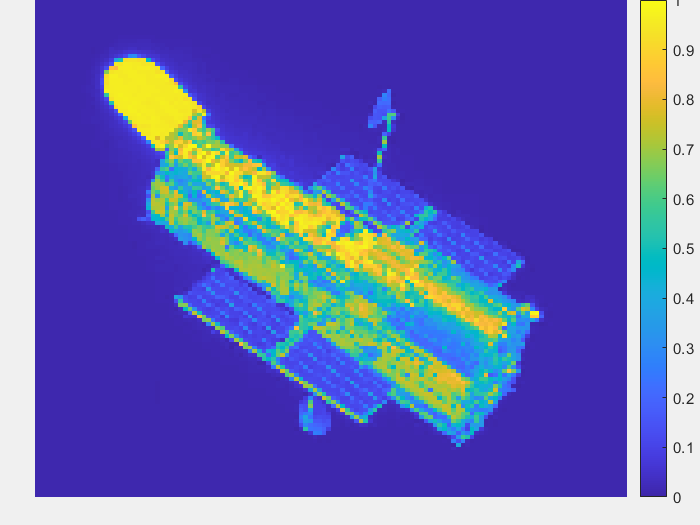}
        \caption{True Image, Example \ref{Example Image Deblurring}.}
        \label{True_image_deblurr}
    \end{subfigure}
    \begin{subfigure}{0.4\textwidth}
        \includegraphics[width=\textwidth]{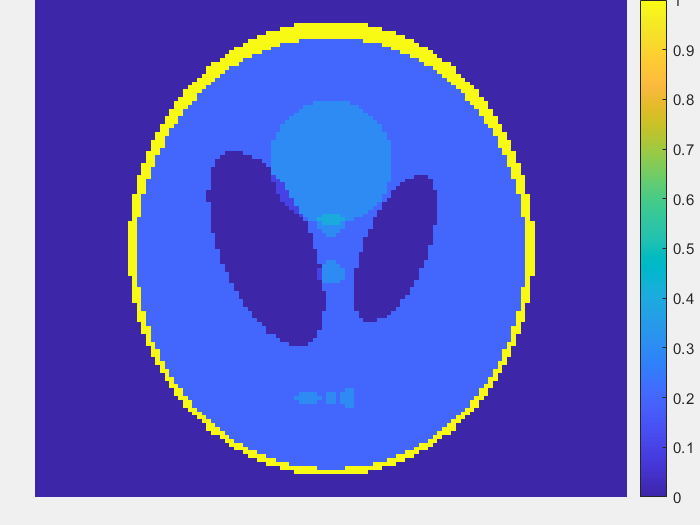}
        \caption{True image, Example \ref{Example Tomography}.}
        \label{True_image_tomo}
    \end{subfigure}    
    \begin{subfigure}{0.4\textwidth}
        \includegraphics[width=\textwidth]{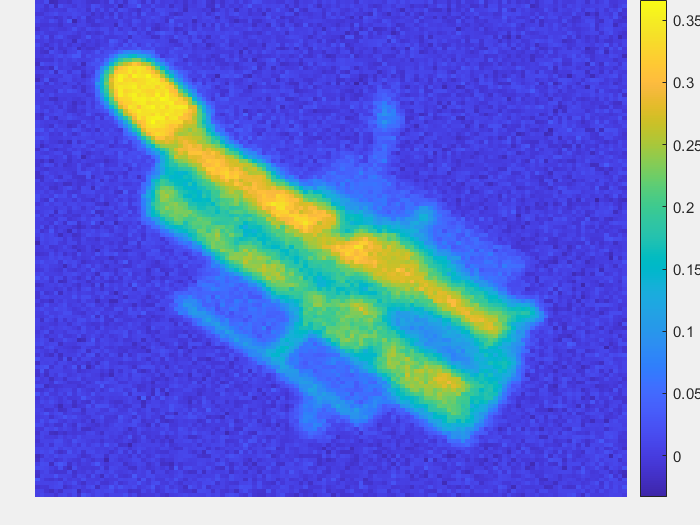}
        \caption{Blurred noisy image}
        \label{blurred_noisy_image}
    \end{subfigure}
    \begin{subfigure}{0.4\textwidth}
        \includegraphics[width=\textwidth]{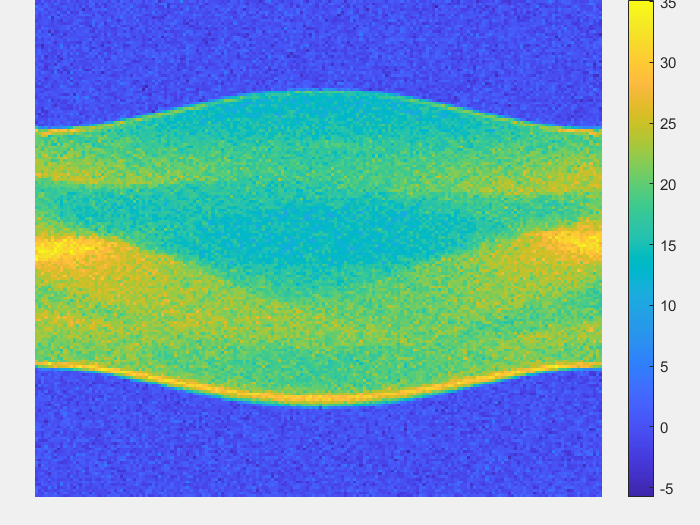}
        \caption{Noisy tomography phantom.}
        \label{Noisy_image_tomo}
    \end{subfigure}
    \begin{subfigure}{0.4\textwidth}
        \includegraphics[width=\textwidth]{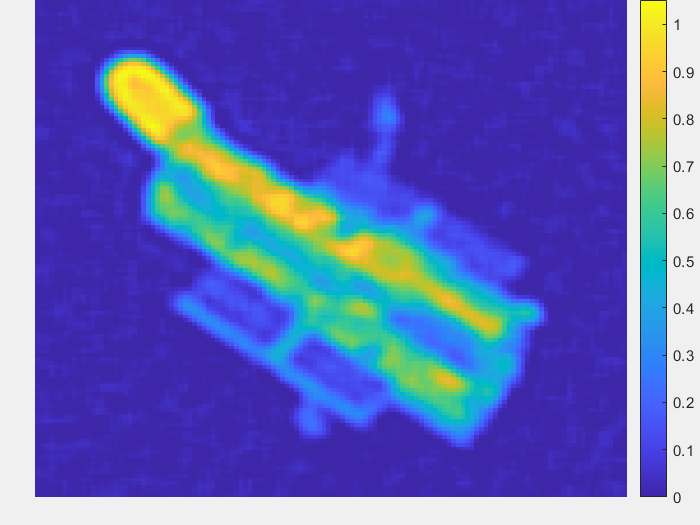}
        \caption{Deblurred image.}
        \label{recovered_deblurred_image}
    \end{subfigure}
    \begin{subfigure}{0.4\textwidth}
        \includegraphics[width=\textwidth]{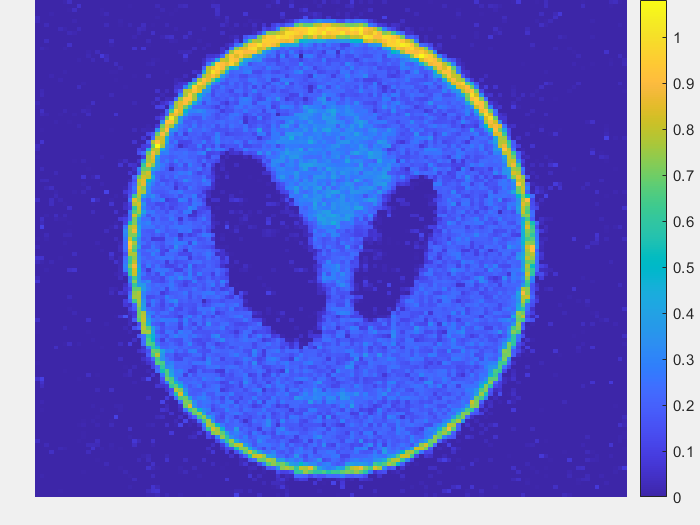}
        \caption{Tomography recovery.}
        \label{recovered_image_tomo}
    \end{subfigure}
    \begin{subfigure}{0.4\textwidth}
        \includegraphics[width=\textwidth]{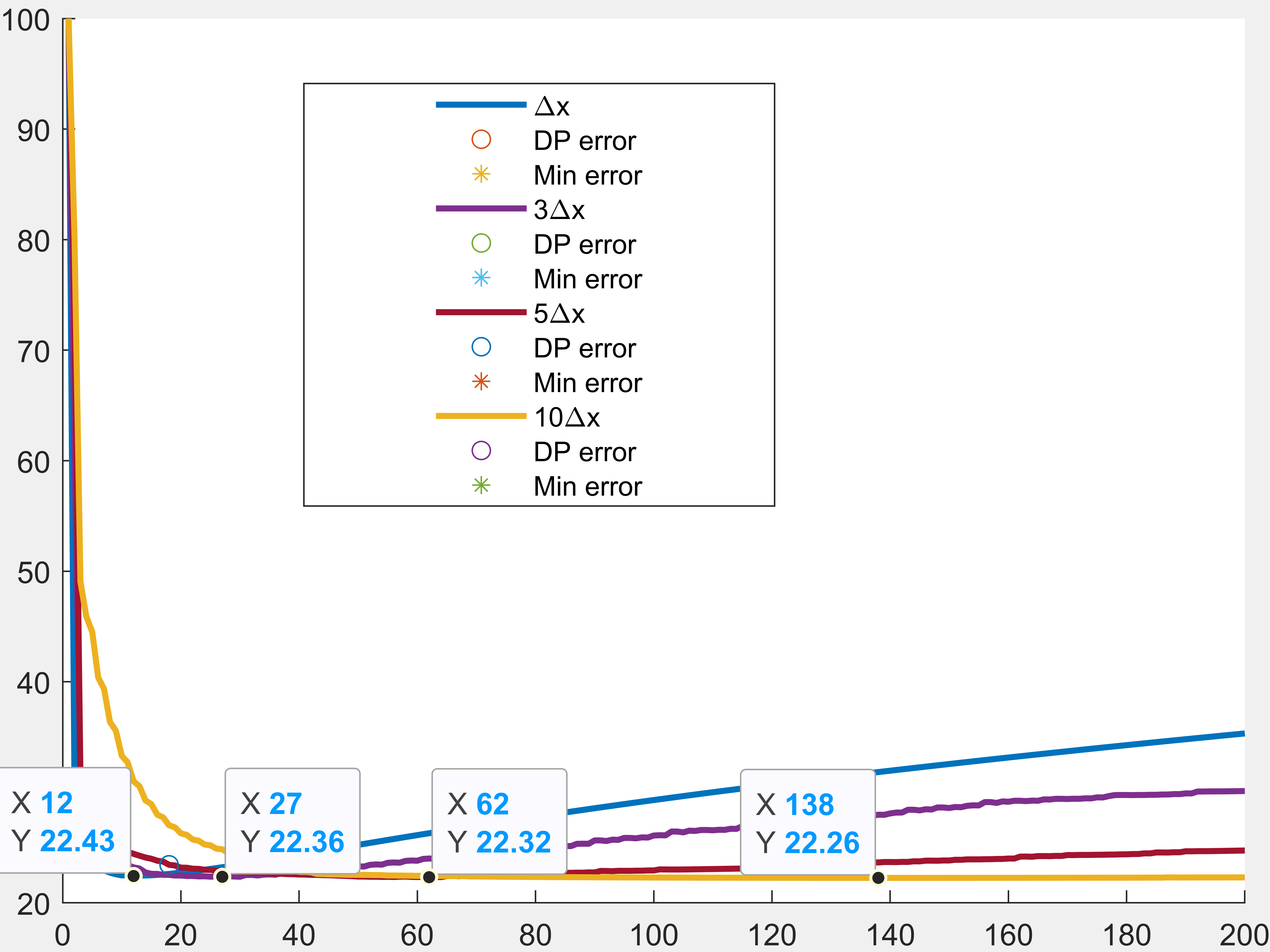}
        \caption{Relative errors descent.}
        \label{Hubble error descent}
    \end{subfigure}
    \begin{subfigure}{0.4\textwidth}
        \includegraphics[width=\textwidth]{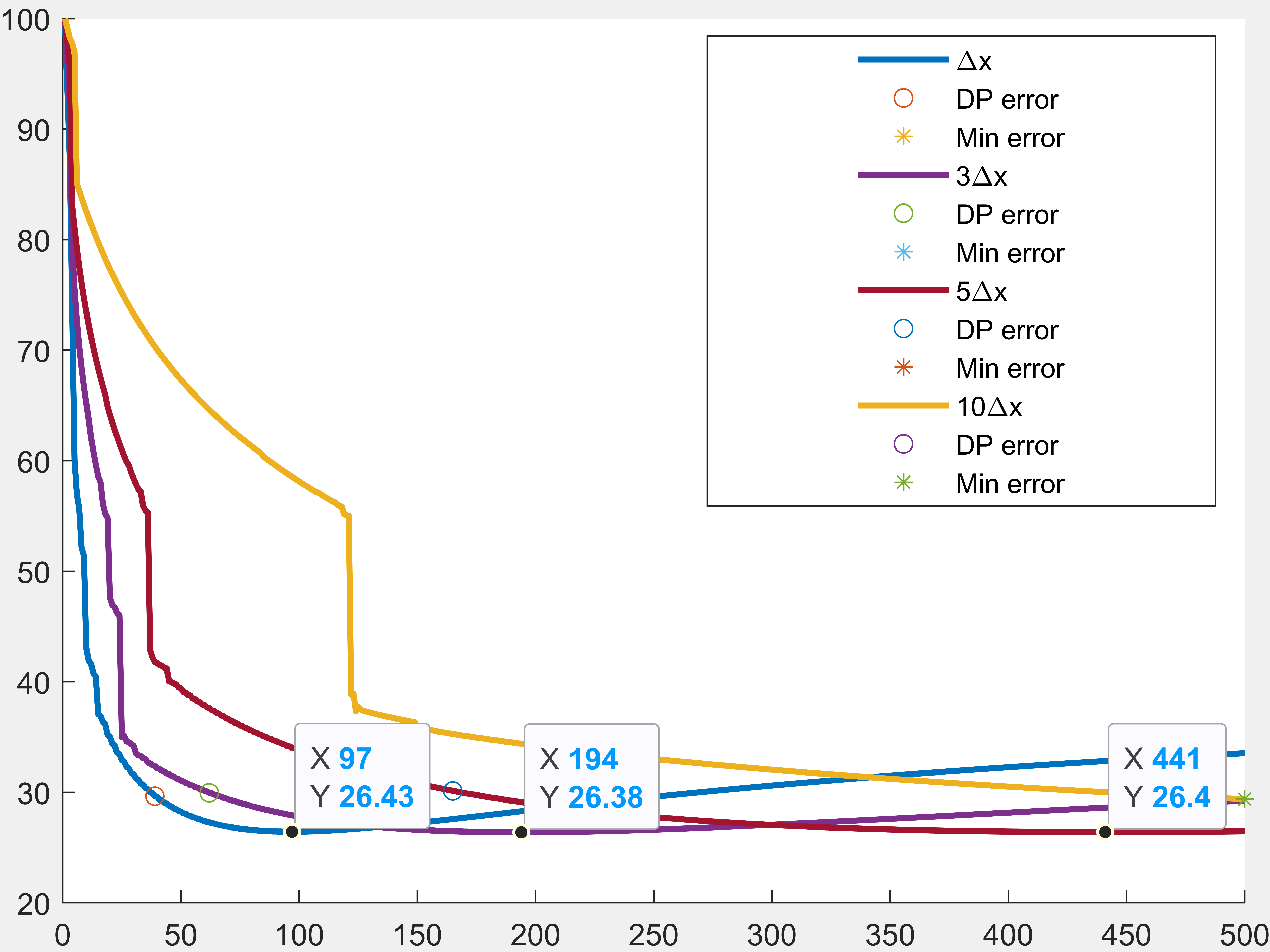}
        \caption{Relative errors descent.}
        \label{Tomography error descent}
    \end{subfigure}    
    \caption{Imaging problems: Example \ref{Example Image Deblurring} and \ref{Example Tomography}. Note that, increasing $\Delta x$ values has a saturating effect on the relative errors.} 
    \label{imaging problems pics}
\end{figure}

\end{example}{}

\begin{example}\label{Example Tomography}\textbf{[Tomography test problem]}\\
In this example we considered the tomography test problem from \cite{Hansen2007}, where the matrix $A \in \mathbb{R}^{n \times n}$ and the vectors $x$, $b \in \mathbb{R}^n$ (where $n=16384$) are generated using the routine, $PRtomo()$, defined in \cite{Hansen_IRtools}. The true image is shown in Figure \ref{True_image_tomo} and the noisy image ($\approx$ 10\% relative noise) for vector $b = Ax$ is shown in Figure \ref{Noisy_image_tomo}. Here, we again performed a constrained minimization ($\psi \geq 0$) during the descent process, as in Example \ref{Example Image Deblurring}. The recovered image is shown in Figure \ref{recovered_image_tomo} and Table \ref{Imaging error comparison} shows an error comparison between our method and some of the standard iterative regularization methods mentioned in Example \ref{Example Image Deblurring}. 
\end{example}{}

\begin{table}[ht]
    \centering
    \begin{tabular}{|p{2.2cm}||p{2.1cm}|p{2.3cm}||p{2.1cm}|p{2.3cm}|}
    \hline
    \multicolumn{5}{|c|}{Relative errors in the recoveries using different regularization methods.}\\
    \hline
    & {Deblurring} & & Tomography &\\
    \hline
    Methods & Relative error & Sparsity (\# 0s) & Relative error & Sparsity (\# 0s)\\
    \hline
    IRell1 & 0.3126 & 0 &  &  \\
    \hline
    IRhtv & 0.3361 & 5732 & 0.5404 & 225 \\
    \hline
    IRcgls & 0.2353 & 0 & 0.3436 & 0 \\
    \hline
    IRfista & 0.2349 & 4058 & 0.3012 & 7505 \\
    \hline
    IRmrnsd & 0.2257 & 1 & 0.2462 & 1 \\
    \hline
    IRenrich & 0.2352 & 0 & 0.3424 & 0 \\
    \hline
    IRirn & 0.3199 & 2035 &  &  \\
    \hline
    IRhy.\_fgmres & 0.3126 & 0 &  &  \\
    \hline
    IRhy.\_gmres & 0.2444 & 0 &  &  \\
    \hline
    IRhy.\_lsqr & 0.2793 & 0 & 1.025 & 0 \\
    \hline
    \hline
    {\color{red} Our ($\nabla_{\mathcal{L}^2-\mathcal{L}^2}$)} & {\color{red} 0.2357} & {\color{red}3797} & {\color{red} 0.2964} & {\color{red} 7563} \\
    \hline
    \end{tabular}
    \caption{Imaging inverse problems: Example \ref{Example Image Deblurring} and Example \ref{Example Tomography}.}
    \label{Imaging error comparison}
\end{table}

\begin{remark}
Again, one can notice (from Figure \ref{Hubble error descent}) that increasing the $\Delta x$ values saturate the relative errors of the recoveries during the descent process. However, if $\Delta x$ value is increased significantly then it will also decrease the rate of the descent process, as can be seen in Figure \ref{Tomography error descent}, especially when dealing with highly discontinuous functions (such as images).
\end{remark}

\section{\textbf{Conclusion and future research}}
In this paper we present a new approach for iterative regularization where instead of using the original noisy gradient we opted a smoother gradient, which is more robust to the noise present in the original data ($g_\delta$). The numerical results show that the method attains lower minimal relative errors during the recovery process, as well as, saturates the semi-convergent nature of the recovery relative errors. One can also notice that when dealing with high noise level data, where the operator $T^*$ fails to smooth the noisy effect of $g_\delta$, having smoother gradients is much more effective than the usual gradient direction. 

In follow up papers we are trying to generalize this regularization method to solve non-linear inverse problems and also incorporate sparsity in the regularization process.

\section*{Acknowledgment}
I am very grateful to Prof. Ian Knowles for his support, encouragement and stimulating discussions throughout the preparation of this paper.

\bibliography{thesisref} 

\begin{thebibliography}{10}

\bibitem{Engl+Hanke+Neubauer}
H.~W. Engl, M.~Hanke, and A.~Neubauer, {\em Regularization of inverse
  problems}, vol.~375 of {\em Mathematics and its Applications}.
\newblock Kluwer Academic Publishers Group, Dordrecht, 1996.

\bibitem{Bakushinsky+Goncharsky}
A.~Bakushinsky and A.~Goncharsky, {\em Ill-posed problems: theory and
  applications}, vol.~301 of {\em Mathematics and its Applications}.
\newblock Kluwer Academic Publishers Group, Dordrecht, 1994.
\newblock Translated from the Russian by I. V. Kochikov.

\bibitem{Groetsch}
C.~W. Groetsch, {\em The theory of {T}ikhonov regularization for {F}redholm
  equations of the first kind}, vol.~105 of {\em Research Notes in
  Mathematics}.
\newblock Pitman (Advanced Publishing Program), Boston, MA, 1984.

\bibitem{Baumeister}
J.~Baumeister, {\em Stable solution of inverse problems}.
\newblock Advanced Lectures in Mathematics, Friedr. Vieweg \& Sohn,
  Braunschweig, 1987.

\bibitem{Morozov_a}
V.~A. Morozov, {\em Methods for solving incorrectly posed problems}.
\newblock Springer-Verlag, New York, 1984.
\newblock Translated from the Russian by A. B. Aries, Translation edited by Z.
  Nashed.

\bibitem{Landweber}
L.~Landweber, ``An iteration formula for fredholm integral equations of the
  first kind,'' {\em American Journal of Mathematics}, vol.~73, no.~3,
  pp.~615--624, 1951.

\bibitem{Hanke_Neubauer_Scherzer}
M.~Hanke, A.~Neubauer, and O.~Scherzer, ``A convergence analysis of the
  landweber iteration for nonlinear ill-posed problems,'' {\em Numerische
  Mathematik}, vol.~72, pp.~21--37, Nov 1995.

\bibitem{Hanke1991}
M.~Hanke, ``Accelerated landweber iterations for the solution of ill-posed
  equations,'' {\em Numerische Mathematik}, vol.~60, pp.~341--373, Dec 1991.

\bibitem{Schock1986}
E.~Schock, ``Semi-iterative methods for the approximate solution of ill-posed
  problems,'' {\em Numerische Mathematik}, vol.~50, pp.~263--271, May 1986.

\bibitem{Bakushinskii1979}
A.~B. Baku\v{s}inski\u{\i}, ``On the principle of iterative regularization,''
  {\em Zh. Vychisl. Mat. i Mat. Fiz.}, vol.~19, no.~4, pp.~1040--1043, 1084,
  1979.

\bibitem{Bakushinsky+Kokurin2004}
A.~B. Bakushinsky and M.~Y. Kokurin, {\em Iterative methods for approximate
  solution of inverse problems}, vol.~577 of {\em Mathematics and Its
  Applications (New York)}.
\newblock Springer, Dordrecht, 2004.

\bibitem{Hanke_2017}
M.~Hanke, {\em Conjugate gradient type methods for ill-posed problems}.
\newblock 01 2017.

\bibitem{Abinash1}
A.~Nayak, ``A new regularization approach for numerical differentiation,'' {\em
  Inverse Problems in Science and Engineering}, vol.~0, no.~0, pp.~1--26, 2020.

\bibitem{Knowles2004}
I.~Knowles, ``Variational methods for ill-posed problems,'' in {\em Variational
  methods: open problems, recent progress, and numerical algorithms}, vol.~357
  of {\em Contemp. Math.}, pp.~187--199, Amer. Math. Soc., Providence, RI,
  2004.

\bibitem{Kaltenbacher_Neubauer_Scherzer}
B.~Kaltenbacher, A.~Neubauer, and O.~Scherzer, {\em Iterative regularization
  methods for nonlinear ill-posed problems}, vol.~6 of {\em Radon Series on
  Computational and Applied Mathematics}.
\newblock Walter de Gruyter GmbH \& Co. KG, Berlin, 2008.

\bibitem{Hansen_MATLAB2020}
P.~C. Hansen, ``Regularization tools: A matlab package for analysis and
  solution of discrete ill-posed problems. version 4.1.,'' 2020.

\bibitem{Hansen_IRtools}
S.~Gazzola, P.~Hansen, and J.~Nagy, ``Ir tools - a matlab package of iterative
  regularization methods and large-scale test problems,'' {\em Numerical
  Algorithms}, 2018.

\bibitem{Hansen2007}
P.~C. Hansen, ``Regularization {T}ools version 4.0 for {M}atlab 7.3,'' {\em
  Numer. Algorithms}, vol.~46, no.~2, pp.~189--194, 2007.

\end{thebibliography}
\bibliographystyle{ieeetr}

\end{document}